\documentclass[a4paper,oneside,11pt]{article}
\usepackage[utf8]{inputenc}
\usepackage{lmodern,stackrel}
\usepackage[T1]{fontenc}
\usepackage{verbatim,bbm}
\usepackage{textcomp}
\usepackage{constants}
\usepackage[english]{babel}
\usepackage{indentfirst}
\usepackage[pdftex]{hyperref}
\usepackage[svgnames]{xcolor}
\usepackage[pdftex]{color,graphicx}
\usepackage{graphicx}
\usepackage{amssymb}
\usepackage{amsmath,mathtools}
\usepackage{epstopdf}
\usepackage{dsfont}
\usepackage{tikz}
\usepackage{mathpazo}

\renewcommand{\leq}{\leqslant}
\renewcommand{\geq}{\geqslant}

\usepackage{amsmath,amsfonts,amssymb,amsthm,mathrsfs,stackrel}

\def\build#1_#2^#3{\mathrel{
\mathop{\kern 0pt#1}\limits_{#2}^{#3}}}

\definecolor{darkred}{rgb}{0.8, 0.0, 0.0}

\renewcommand{\leq}{\leqslant}
\renewcommand{\geq}{\geqslant}

\def\build#1_#2^#3{\mathrel{
\mathop{\kern 0pt#1}\limits_{#2}^{#3}}}

\newcommand{\E}{\mathcal{E}}

\newcommand{\Z}{{\mathbb{Z}}}
\newcommand{\LL}{{\mathbb{L}}}
\newcommand{\N}{{\mathbb{N}}}

\renewcommand{\L}{{\Lambda}}
\newcommand{\V}{{\mathcal{V}}}
\renewcommand{\P}{{\mathbb{P}}}

\newcommand{\xgt}{\Xi_{\Lambda,t,r}}

\newcommand{\prt}[1]{\left(#1\right)}

\theoremstyle{plain}
\newtheorem{thm}{Theorem}
\newtheorem{corollary}{Corollary}
\newtheorem{proposition}[corollary]{Proposition}

\newtheorem{lemma}{Lemma}

\theoremstyle{definition}

\newtheorem{remark}{Remark}

\newconstantfamily{rog}{
symbol=c,
}

\begin{document}
\thispagestyle{empty}
\baselineskip=28pt
\vskip 5mm
\begin{center}
\LARGE{\textbf{Constrained-degree percolation  in  random environment}}
\end{center}
\baselineskip=14.5pt
\vskip 15mm

\begin{center}\large
R\'{e}my Sanchis\footnote{Departamento de Matem\'{a}tica, Universidade Federal de Minas Gerais, Brazil, \href{mailto:rsanchis@mat.ufmg.br}{rsanchis@mat.ufmg.br}; https://orcid.org/0000-0002-4472-3628} , Diogo C. dos Santos\footnote{Departamento  de Estat\'{\i}stica, Universidade Federal de Minas Gerais, Brazil, \href{mailto:diogoc@.ufmg.br}{diogoc@ufmg.br;} https://orcid.org/0000-0002-3345-6851
} 
and Roger W. C. Silva 
 \footnote{Corresponding Author - Departamento de Estat\'{\i}stica,
 Universidade Federal de Minas Gerais, Brazil, \href{mailto:rogerwcs@ufmg.br} {rogerwcs@ufmg.br;} https://orcid.org/0000-0002-6365-9211
}
\\ 

\vskip 10mm

\end{center}

\begin{abstract}

We consider the Constrained-degree percolation model in random environment on the square lattice. In this model,  each vertex $v$ has an independent random constraint ${\kappa}_v$ which takes the value $j\in \{0,1,2,3\}$ with probability $\rho_j$. Each edge $e$ attempts to open at a random uniform time $U_e$ in $[0,1]$, independently of all other edges. It succeeds if at time $U_e$ both  its end-vertices have degrees strictly smaller than their respectively attached constraints. We show that this model undergoes a non-trivial phase transition when $\rho_3$ is sufficiently large. The proof consists of a decoupling inequality, the continuity of the probability for local events, and a coarse-graining argument.

\end{abstract}
{\footnotesize Keywords: exponential decay of correlations; constrained-degree percolation, renormalization, continuity }

\section{Introduction}

The Constrained-degree percolation model was introduced in \cite{Te}. It is a dependent continuous-time percolation process that evolves according to the following dynamics: after a  uniformly distributed random time, each edge attempts to open and it succeeds if at this random time both its end-vertices have degrees strictly smaller than some fixed integer number $\kappa$. Once an edge becomes open it remains open. In contrast to classical dynamical Bernoulli percolation, at any fixed time $t\in(0,1]$, the model exhibits dependencies of all orders. 
Other examples of models with infinite range dependence include Voronoi percolation \cite{BR}, Boolean percolation \cite{ATT}, and Majority percolation \cite{AM}. In such models, a decoupling technique must be carried out to understand the underlying structure of the model. Indeed, decoupling tools together with multiscale renormalization schemes constitutes a very powerful mechanism to approach dependent systems.

There is a growing literature on models with some type of local constraint.
For instance, in \cite{GJ} the Erd\"os-R\'enyi random graph $G_{n,p}$ is studied conditioned on the event that all vertex degrees belong to some subset $\mathcal{S}$ of the non-negative integers. Subject to a tricky hypothesis on $\mathcal{S}$, the authors obtain a condition for the existence (or not) of a giant component. Rigorous mathematical results on the so-called 1-2 Model, a statistical mechanics model on the hexagonal lattice in which the degree of each vertex is constrained to be 1 or 2, are reviewed in \cite{GL}. In \cite{Ga} the authors study Eulerian percolation on the square lattice, a classical Bernoulli bond percolation model with parameter $p$ conditioned on the fact that every site is of even degree. See \cite{HL,ZL} for other examples of constraint models.



Denote by $\mathbb{L}^2=(\mathbb{Z}^2,\E^2)$ the usual square lattice. In \cite{LSSST}, the authors show that the Constrained-degree percolation model on $\mathbb{L}^2$ admits a non-trivial phase transition when $\kappa_v=3$ for all $v\in\mathbb{Z}^2$. They also show that when $\kappa_v\leq2$ for all $v\in\mathbb{Z}^2$, there is no percolation, even at time 1. A natural question is the following: what happens if we allow some positive density of vertices to have constraints strictly lesser than three? Would the model still exhibit a non-trivial critical phenomenon?  

For instance, in \cite{FSZD} the authors simulate various scenarios for the Constrained-degree percolation process on $\mathbb{L}^2$. Their numerical results suggest that the model with uniform constraints 2 and 3 percolates at some time around 0.728. On the other hand, the model with uniform constraints 1, 2, and 3 does not percolate at any time. 

In recent years, special attention has been given to models where the underlying fixed structure of the graph is replaced by a random environment. Pioneering work in this direction is the paper of McCoy and Wu (see \cite{MW}), where the Ising model with impurities is introduced, allowing each row of vertical bonds to vary randomly from row to row with a prescribed probability function. The authors investigate the effect of random impurities on the phase transition of the model. Since the seminal paper \cite{MW}, the question of whether a random environment could affect the phase transition has been investigated in different contexts, including percolation \cite{HO,JMP}, oriented percolation \cite{KSV}, Ising model \cite{CK,CKP}, and the contact process \cite{BDS,K}. 

Motivated by this we introduce a variant of the model studied in \cite{LSSST}, allowing each vertex of the square lattice to receive a constraint value on the set $\{0,1,2,3\}$ according to an $i.i.d.$ law. Our main result is the existence of a non-trivial critical phenomenon for this model when the density of sites with constraint three is sufficiently close to one. The model is formally defined in the next section.

\subsection{The Model}\label{model}

We introduce the model on $\mathbb{L}^2=(\mathbb{Z}^2,\E^2)$.  Let $\kappa=\{\kappa_v\}_{v\in\mathbb{Z}^2}$ be a sequence of $i.i.d.$ random variables taking values $j\in\{0,1,2,3\}$ with probability $\rho_j$. The corresponding product measure is denoted by $\P_{\rho}$.

A temporal configuration is introduced with a sequence $\{U_e\}_{e\in\E^2}$ of {\em i.i.d.} uniform random variables in $[0,1]^{\E^2}$, independent of the sequence  $\kappa=\{\kappa_v\}_{v\in\mathbb{Z}^2}$ (sometimes we call $\{U_e\}_{e\in\E^2}$ a configuration of clocks).
We consider a continuous-time percolation configuration \begin{equation*}\omega_t:[0,1]^{\E^2}\times\{0,1,2,3\}^{\mathbb{Z}^2}\longrightarrow \{0,1\}^{\E^2},\end{equation*}
constructed as follows: at time $t=0$ all edges are closed. Denote by $\deg(v,t)$ the degree of vertex $v$ in $\omega_t$ (note that $\deg(v,0)=0$ for all $v\in\mathbb{Z}^2$).  Each edge $e=\langle u,v\rangle$ opens
 at time $U_e$ if both $\deg(u,U_e)<\kappa_{u}$ and $\deg(v,U_e)<\kappa_{v}$. Denote by $\omega_t(e)$ the configuration at edge $e$. We say $e$ is \textit{t-open} or \textit{t-closed} according to $\omega_t(e)=1$ or $\omega_t(e)=0$; also let $\P_{\rho,t}$ denote the law of $\omega_t$. Formally, we can express $\omega_t$ as the product of a collection of indicator functions, \textit{i.e.}, if $e=\langle u,v\rangle$, then
 \begin{equation*}
 \omega_t(e)=\mathbbm{1}_{\{U_e\leq t\}}\times\mathbbm{1}_{\{\deg(u,U_e)<\kappa_{u}\}}\times\mathbbm{1}_{\{\deg(v,U_e)<\kappa_{v}\}}.
\end{equation*}
Using Harris' construction  (see discussion after Theorem \ref{main}), it is straightforward  to show
that  $\omega_t$ is well defined for almost all sequences $U=\{U_e\}_{e\in\E^2}$ and $\kappa=\{\kappa_v\}_{v\in\mathbb{Z}^2}$ and all $t\in[0,1]$. 

Denote by $\mathbb{P}$ the product of Lebesgue measures governing the sequence $\{U_e\}_{e\in\E^2}$. It is clear that the bond percolation process defined above is governed by the pushforward product law
\begin{equation*}
\mathbb{P}_{\rho,t}(A)
=(\mathbb{P}\times\P_{\rho})(\omega_t^{-1}(A)),
\end{equation*}
 for any event $A\subset\{0,1\}^{\E^2}$.

Our main result is the following:
 
\begin{thm}\label{PhaseTransition}
Consider the Constrained-degree percolation model on $\mathbb{L}^2$ with $\rho=(\rho_0, \rho_1, \rho_2,\rho_3)$. There exists a constant  $\eta>0$ such that
\begin{equation*}
\mathbb{P}_{\rho,t}(0\longleftrightarrow\infty)>0,\,\,\,\,\,\,\,\,\, \forall\, (\rho_3,t)\in(1-\eta,1]^2.
\end{equation*}
\end{thm}

\subsection{Organization}

We shall implement a renormalization scheme, which is done in two steps: the triggering argument and the induction step.  In Section \ref{decoup_sec}, we prove a decoupling inequality which will be useful for the induction step of the construction.  In Section \ref{contin_sec}, we show that $\mathbb{P}_{\rho,t}(A)$ is a continuous function of $(\rho,t)$ for any local event $A$, a fact that will be important for  proving the triggering argument. In Section \ref{main_thm} we use a coarse-graining argument to conclude the proof of Theorem \ref{PhaseTransition}.  We stress that the results of Section \ref{decoup_sec} and Section \ref{contin_sec} are valid in a more general context.  

\section{Decay of Correlations}\label{decoup_sec}
 
 In this section, we introduce the model on the hypercubic lattice $\mathbb{\mathbb{L}}^d=(\Z^d,\E^d)$. All the results and discussions in Section \ref{decoup_sec} and Section \ref{contin_sec} hold in this more general framework. In Section \ref{main_thm}, we return to the model on the square lattice $\LL^2$ to prove Theorem \ref{PhaseTransition}.
 
 Let $\kappa=\{\kappa_{v}\}_{v\in\Z^d}$ be a sequence of $i.i.d.$ random variables taking value $j\in\{0,1,\dots,2d-1\}$ with probability $\rho_j$.
The corresponding product measure is denoted by $\P_{\rho}$. As before, a temporal configuration is introduced with a sequence $\{U_e\}_{e\in\E^d}$ of $i.i.d.$ uniform random variables on $[0,1]^{\E^d}$ with corresponding product measure $\P$, independent of the sequence  $\kappa=\{\kappa_v\}_{v\in\Z^d}$.  The percolation configuration is denoted by 
\begin{equation*}\omega_t:[0,1]^{\E^d}\times\{0,1,\dots,2d-1\}^{\Z^d}\longrightarrow \{0,1\}^{\E^d},
\end{equation*} and constructed as follows: at time $t=0$ all edges are closed. Each edge $e=\langle u,v\rangle$ opens
 at time $U_e$ if both $\deg(u,U_e)<\kappa_{u}$ and $\deg(v,U_e)<\kappa_{v}$. Denote by $\omega_t(e)$ the percolation configuration at edge $e$ and by
\begin{equation*}\label{prod_meas}
\mathbb{P}_{\rho,t}(A)
=(\P\times \P_{\rho})(\omega_t^{-1}(A))
\end{equation*}
the law of $\omega_t$, for any event $A\subset\{0,1\}^{\E^d}$.

Let us introduce some notation. We denote by $\delta(u,v)$ the graph distance between vertices $u,v\in\Z^d$.  Given $\L,\tilde{\L}\subset \Z^d$, let $\delta(\tilde{\Lambda},\Lambda)=\inf\{\delta(x,y)\colon x\in\tilde{\Lambda},\, y\in\Lambda\}$ and write $B_{r}(\Lambda)=\{x\in\Z^d\colon \delta(x,\Lambda)\leq r\}$ for the ball of radius $r$ centered at $\Lambda$, $r\in\mathbb{N}$.
If $\Lambda$ is unitary, say $\Lambda=\{v\}$, we write $B_{r}(v)$ instead of $B_{r}(\{v\})$. Denote by $\partial\Lambda$, $\partial^s\Lambda$, and 
$\partial^e\Lambda$  the vertex, external vertex, and external edge boundaries of $\L$, respectively, that is,  
\begin{equation*}\partial\Lambda=\{u\in \Lambda\colon  \exists v\in {\Lambda}^c\mbox{ such that } \delta(u,v)=1\},
\end{equation*}
\begin{equation*}\partial^s\Lambda=\{u\in \Lambda^c\colon  \exists v\in {\Lambda}\mbox{ such that } \delta(u,v)=1\},
\end{equation*}
\begin{equation*}\partial^e\Lambda=\{\langle u,v\rangle\in\E^d\colon\ u\in \Lambda \mbox{ and } v\in {\Lambda}^c\}.
\end{equation*}

Given a collection of sites $\Gamma\subset\Z^d$,  we will use $\E(\Gamma)$ to denote the set of {\it edges} in the graph generated by $\Gamma$, i.e., the set of edges with both endpoints in $\Gamma$.

The restriction of a configuration $\omega$ to a set $K\subset\mathcal{E}^d$ is $\omega|_K=\{\omega(i)\colon i\in K\}.$ Define the cylinder set
$[\omega]_{K}=\{\omega'\colon\omega'|_K=\omega|_K\}.$
We say that the event $A$ \textit{lives on} a set $K$ if $\omega\in A$ and $\omega'\in [\omega]_{K}$ imply $\omega'\in A$.  Given $K\subset\E^d$, we write $\mathcal{H}_{K}=\{A\colon A\mbox{ lives on }K\}$. We say $A$ is a \emph{local event} if $A$ lives on some finite set $K$.

We will prove the  decoupling inequality below.

\begin{thm}\label{main}
Let $\Lambda_1$ and $\Lambda_2$ be finite disjoint subsets of $\Z^d$. Assume $A_1\in \mathcal{H}_{\E(\Lambda_1)}$ and $A_2\in \mathcal{H}_{\E(\Lambda_2)}$. There exists $\psi(d)>0$ such that
\begin{equation*}
  \left|   \mathbb{P}_{\rho,t}(A_1\cap A_2)-\mathbb{P}_{\rho,t}(A_1)\mathbb{P}_{\rho,t}(A_2)    \right|
     \leq c_3\left(|\partial\Lambda_1|+|\partial\Lambda_2|\right)e^{-\psi \delta(\Lambda_1,\Lambda_2)},
\end{equation*}
for some constant $c_3(d)>0$, and for all $\rho=(\rho_0,\dots,\rho_{2d-1})$ and $t\in[0,1]$.
\end{thm}

Consider the following definitions related to the space of clocks: given $a,b\in[0,1]$ and a sequence of times $U=\{U_e\}_{e\in \E^d}$, denote by $C_{a,b}(x)$ the set of vertices connected to $x\in\Z^d$ by a path of edges whose clocks ring in the time interval $(a,b]$.
That is, $v\in C_{a,b}(x)$ if, and only if, there exists an alternating
sequence $x=x_0,e_0,x_1,e_1,\dots,e_{n-1},x_n=v$ of distinct vertices $x_i$ and edges $e_i=\langle x_i,x_{i+1}\rangle$ such that $a<U_{e_i}\leq b$. Note that $C_{a,b}(x)$ can be identified with the open cluster of vertex $x$ in an independent Bernoulli percolation model with parameter $b-a$. 
If $\Lambda\subset \Z^d$, we write $C_{a,b}(\Lambda)$ for the union of  clusters of vertices of $\Lambda$, \textit{i.e.}, 
$C_{a,b}(\Lambda)=\bigcup_{x\in \Lambda}C_{a,b}(x).$ Given $\Gamma\subset\Z^d$, we write $C_{a,b}^{\Gamma}(x)$ for the set of vertices connected to $x\in\Z^d$ by a path of edges whose clocks ring in the time interval $(a,b]$, and that uses only edges with both end-points in $\Gamma$. 

Let us briefly sketch the proof of Theorem \ref{main} in the particular case $t<p_c(\mathbb{L}^d)$, where $p_c(\mathbb{L}^d)$ is the critical threshold for independent Bernoulli bond percolation on $\mathbb{L}^d$. Given a finite set $\Lambda\subset \Z^d$, we will show (see Lemma \ref{independencia}) that on the event that $C_{0,t}(\Lambda)$ is confined to some ball of radius $r$, say $B_r(\L)$, the occurrence of any event that lives on $\Lambda$ depends only on the variables $U_e$ on edges of $\E(B_{r+1}(\Lambda))$ and the variables $\kappa_v$ with $v\in B_{r+1}(\Lambda)$. This fact and the exponential decay of the radius of $C_{0,t}(\Lambda)$ (which follows from a coupling with Bernoulli percolation of parameter $t$) gives us, with a little extra work, the necessary tools to prove the decoupling when $t<p_c(\Z^d)$. It turns out that this argument is no longer enough when $t\geq p_c(\mathbb{L}^d)$. In this case, pick $m\in\{1,...,2d\}$ such that  $t\in(\frac{m-1}{2d},\frac{m}{2d}]$ and write
\begin{equation}\label{influence}\mathcal{I}_t(\Lambda)= C_{0,\frac{1}{2d}}\left( C_{\frac{1}{2d},\frac{2}{2d}}\left(\ldots \left(C_{\frac{m-1}{2d},t}(\Lambda)\right) \right)\right).
\end{equation}

The set $\mathcal{I}_t(\L)$ is seen as a composition of clusters containing $\Lambda$. The first layer is the set of vertices connected to $\Lambda$ by a path of edges whose clocks ring between times $\frac{m-1}{2d}$ and $t$; the second layer is the set of vertices connected to the first layer by a path of edges whose clocks ring between times $\frac{m-2}{2d}$ and $\frac{m-1}{2d}$, and so on. The choice of intervals of length $1/2d$ is arbitrary, but $2d$ is a nice integer such that $1/2d < p_c(\Z^d)$, and so it gives the exponential tail decay of the radius of each layer in $\mathcal{I}_t(\Lambda)$.

 The random set in \eqref{influence} can be used to show that the model is well defined. In our case, the classical Harris' argument reduces to the observation that, for any $t\in[0,1]$, the state $\omega_t(e)$ of any edge $e\in\E^d$ is determined by checking only a finite number of random variables.  Indeed, Lemma \ref{prop} below says that $\mathcal{I}_t(\L)$ is almost surely finite and Lemma \ref{independencia} states that if $\mathcal{I}_t(\L)$ is limited to some ball $B_r(\L)$, then the occurrence of any event that lives on $\Lambda$ depends only on the variables $U_e$ on edges of $\E(B_{r+1}(\Lambda))$ and the variables $\kappa_v$ with $v\in B_{r+1}(\Lambda)$ (the careful reader will notice that set of edges and vertices whose clocks and constraints actually affect the occurrence of any event living on $\L$ is a subset of such ball). Therefore, to determine the state of each edge at any time, it is enough to check the configuration of clocks and constraints in some almost surely finite region.  Lemma \ref{exponential_decay} below gives quantitative estimates for the size of such a region.

In the remainder of this section, we present the three lemmas that will help us to prove Theorem \ref{main} and their proofs. In what follows, $\mathbb{E}(\cdot)$ denotes the expectation with respect to the product measure $\mathbb{P}$ on the space of clocks.
\begin{lemma}\label{prop} There exists a constant $c_1=c_1(d)<\infty$ such that, for all $v\in\Z^d$,
\begin{equation*}
\sup_{t\in[0,1]}\mathbb{E}\left[|\mathcal{I}_t(v)|\right]\leq c_1.
\end{equation*}
\end{lemma}

\begin{proof} 

Since $\mathcal{I}_t(v)\subset \mathcal{I}_1(v)$ for all $t\in[0,1]$, it suffices to show that $\mathbb{E}\left[|\mathcal{I}_1(v)|\right]\leq c_1$ for all $v\in\Z^d$.  
Define the auxiliary sets
\begin{equation*}\mathcal{D}_j(v)=C_{\frac{j}{2d},\frac{j+1}{2d}}\left(\ldots\left( C_{\frac{2d-1}{2d},1}(v)\right)\right),\,\,\,j\in\{0,1,\dots,2d-1\}.
\end{equation*}
Note that $\mathcal{D}_1(v)$ is obtained from $\mathcal{I}_1(v)$ by deleting the outermost cluster of $\mathcal{I}_1(v)$.
For each $v\in \Z^d$, it holds 
\begin{align*}
\mathbb{E}[|\mathcal{I}_1(v)|]&=\mathbb{E}\left[\mathbb{E}\left[|\mathcal{I}_1(v)|\middle| \mathcal{D}_1(v)\right]\right]=\mathbb{E}\left[\mathbb{E}\left[|C_{0,\frac{1}{2d}}\left(\partial\mathcal{D}_1(v)\right)|\middle|\mathcal{D}_1(v)\right]\right].
\end{align*}

For a given $\Gamma\subset\Z^d$, observe that the event $\{D_1(v)=\Gamma\}$ gives information on the clocks of edges in $\E(\Gamma)\cup\partial^e\Gamma$, but is independent of the clocks on edges in $\E(\Gamma^c)$. 
Consider the random sets 
\begin{equation}\label{fronteira}X_i(\Gamma)=\left\{e\in \partial^e\Gamma: U_e\in \left(\frac{i}{2d},\frac{i+1}{2d}\right]\right\},\,\, i=0,1.\dots,2d-1,
\end{equation}
and note that, for fixed $\Gamma$, the function $\mathbb{E}\left[|C_{0,\frac{1}{2d}}(\partial\Gamma)|\middle|\mathcal{D}_1(v)=\Gamma;X_0(\Gamma)\right]$ is maximized when all the clocks on edges of $\partial^e\Gamma$ ring in $(0,1/2d)$. Hence,
\begin{align*} \mathbb{E}\left[\mathbb{E}\left[|C_{0,\frac{1}{2d}}(\partial\Gamma)|\middle|\mathcal{D}_1(v)=\Gamma;X_0(\Gamma)\right]\right]&\leq \mathbb{E}\left[|C_{0,\frac{1}{2d}}(\partial\Gamma)|\middle|\mathcal{D}_1(v)=\Gamma; X_0(\Gamma)=\partial^e\Gamma\right]\\&
\leq \mathbb{E}\left[|\partial\Gamma|+|C_{0,\frac{1}{2d}}^{\Gamma^c}(\partial^s\Gamma)|\middle|\mathcal{D}_1(v)=\Gamma; X_0(\Gamma)=\partial^e\Gamma\right]\\
&= |\partial\Gamma|+\mathbb{E}\left[|C_{0,\frac{1}{2d}}^{\Gamma^c}(\partial^s\Gamma)|\right]\leq |\partial\Gamma|+\mathbb{E}\left[|C_{0,\frac{1}{2d}}(\partial^s\Gamma)|\right].
\end{align*}

A classical counting argument in percolation gives $\mathbb{E}\left[|C_{0,\frac{1}{2d}}(u)|\right]\leq (1+2d)$, for any $u\in\Z^d$. This leads to
\begin{equation*}\mathbb{E}\left[|C_{0,\frac{1}{2d}}(\Gamma)|\middle|\mathcal{D}_1(v)=\Gamma\right]\leq |\partial\Gamma|+|\partial^s\Gamma|(1+2d)\leq |\partial\Gamma|(1+2d(1+2d))\leq 4d(d+1)|\Gamma|,
\end{equation*}
and we obtain 
\begin{equation}\label{exp_bound}
\mathbb{E}[|\mathcal{I}_1(v)|]\leq 4d(1+d)\mathbb{E}[|\mathcal{D}_1(v)|].
\end{equation}

Consider the events $\{\mathcal{D}_j(v)=\Gamma\}$, for some deterministic set $\Gamma$. If we condition on the event that all edges in $\partial^e\Gamma$ ring in the interval $(\frac{j}{2d},\frac{j+1}{2d})$, then the same reasoning leading to \eqref{exp_bound}  yields 
\begin{equation*}\mathbb{E}[|\mathcal{D}_j(v)|]\leq 4d(1+d)\mathbb{E}[|\mathcal{D}_{j+1}(v)|],\,\,\,j\in\{1,\dots,2d-2\}.
\end{equation*}
Therefore, 
\begin{equation*}\mathbb{E}[|\mathcal{I}_1(v)|]\leq [4d(1+d)]^{2d}=c_1(d),
\end{equation*}
and the proof is complete.

\end{proof}

For fixed $\L\subset\Z^d$, $t\in[0,1]$ and $r\in \N$, define the clock event
\begin{equation*}\Xi_{\Lambda,t,r}=\left\{U\colon \mathcal{I}_t(\Lambda)\subset B_{r}(\Lambda) \right\}.
\end{equation*}
When $t=1$ we write $\Xi_{\Lambda,r}$ instead of $\Xi_{\Lambda,1,r}$. In what follows, we abuse notation and write $\omega_t^{-1}(A)\cap\Xi_{\Lambda,t,r}$ for the set of all sequences $\kappa=\{\kappa_v\}_{v\in\Z^d}$ and $U=\{U_e\}_{e\in\mathcal{E}^d}$ that induce the occurrence of $A$ and such that $U\in\Xi_{\Lambda,t,r}$.


\begin{lemma}\label{independencia}
Let $\Lambda\subset \Z^d$ be a finite set. If $A\in\mathcal{H}_{\E(\L)}$, then the event 
$\omega_t^{-1}(A)\cap\Xi_{\Lambda,t,r}$ depends only on the collection of random variables $U_e$ with $e\in \E(B_{r+1}(\Lambda))$ and the random variables $\kappa_v$ with ${v\in B_{r+1}(\L)}$.
\end{lemma}

\begin{proof}
Fix $(\kappa,U)\in\omega_t^{-1}(A)\cap \xgt$. Given a collection of constraints $\Tilde{\kappa}\in\{0,\dots,2d-1\}^{\Z^2}$ and clocks $\Tilde{U}\in[0,1]^{\mathcal{E}^2},$ we consider, with some abuse of notation, the new configurations 
\begin{equation*}(\kappa\times\tilde{\kappa})_v=\left\{\begin{array}{ll}
\kappa_v&\mbox{if}\quad v\in  B_{r+1}(\L)\,,\\
\tilde{\kappa}_v &\mbox{if}\quad v\notin  B_{r+1}(\L),
\end{array}\right.
\end{equation*}
and
\begin{equation*}(U\times\tilde{U})_e=\left\{\begin{array}{ll}
U_e&\mbox{if}\quad e\in \E(B_{r+1}(\Lambda)) ,\\
\tilde{U}_e &\mbox{if}\quad e\notin \E(B_{r+1}(\Lambda)) .
\end{array}\right.
\end{equation*}

Denote by $\tilde{\omega}_t$ the percolation configuration at $\kappa\times\tilde{\kappa}$ and $U\times\tilde{U}$ and by $\omega_t$ the configuration at $\kappa$ and $U$. The first thing to observe is that $U\times\Tilde{U}\in\xgt$. Indeed, let $f\in\partial^e\mathcal{I}_t(\Lambda)$ and observe that $U_f>U_e$, for all $e\in \E(\mathcal{I}_t(\Lambda))$ sharing a vertex with $f$; otherwise $f$ would be in $\E(\mathcal{I}_t(\Lambda))$. Besides, since $U\in\xgt$, all edges on $\partial^e\mathcal{I}_t(\Lambda)$ necessarily have both  end-vertices inside the ball $B_{r+1}(\Lambda)$. This prevents that any change on the clocks outside $\E(B_{r+1}(\Lambda))$ interferes in the occurrence of the event $\xgt$. It remains to show that
\begin{equation}\label{O_ponto}
\tilde{\omega}_t|_{\Lambda}=\omega_t|_{\Lambda}.
\end{equation}

We will actually prove the stronger claim that $\omega_t(e)=\tilde{\omega}_t(e)$ for all $e\in\E(\mathcal{I}_t(\Lambda))$. This fact clearly gives the equality on $\eqref{O_ponto}$. Let $\{e_1,\ldots,e_m\}$ be an enumeration
of the edges with both endpoints in $\mathcal{I}_t(\Lambda)$ such that $U_{e_1}<U_{e_2}<\ldots<U_{e_m}$, and
denote $e_i=\langle x_i,y_i\rangle$. Note that $(U\times\tilde{U})_{e_1}=U_{e_1}<U_{f}=(U\times\tilde{U})_{f}$ 
for all $f\sim e_1$. This follows because if $f\in\E(\mathcal{I}_t(\L))$, then its clock must have ring after $U_{e_1}$. If $f\notin\E(\mathcal{I}_t(\L))$, then by the argument on the last paragraph we must have $U_f>U_{e_1}$. Hence $\tilde{\omega}_t(e_1)=\omega_t(e_1)=1$, unless $\kappa(x_1)=0$ or $\kappa(y_1)=0$, in which case $\tilde{\omega}_t(e_1)=\omega_t(e_1)=0$.

Assume that $\tilde{\omega_t}(e_k)={\omega_t}(e_k)$, $k=1,\dots,i-1$. Since $e_i=\langle x_i,y_i\rangle\in \E(B_{r+1}(\Lambda))$, 
it holds that $(U\times\tilde{U})_{e_i}=U_{e_i}$. Let $s=U_{e_i}$. Consider the set $H$ of edges incident to vertex $x_i$ and observe that $H\subset \E(B_{r+1}(\Lambda))$. Thus, $(U\times \tilde{U})_f=U_f$ for all $f\in H$.   We can partition $H$ into two disjoint sets $H_1$ and $H_2$ in such a way that $f\in H_1$ if $U_f<s$ and $f\in H_2$ if $U_f>s$. If $f\in H_1$, then necessarily $f\in\E(\mathcal{I}_t(\Lambda))$ and thus, by the induction hypothesis, we have $\tilde{\omega}_s(f)=\omega_s(f)$. If $f\in H_2$,
then its clock rings after time $s$. Therefore, it is closed at time $s$, and
hence $\tilde{\omega}_s(f)=\omega_s(f)=0$. The same argument works for $y_i$, and we have shown that the degrees of $x_i$ and $y_i$ are the same on both configurations by time $s$, that is, $\tilde{\omega}_t(e_i)=\omega_t(e_i)$.

\end{proof}

We will need the following definition: given a set $\Lambda\subset\Z^d$ and a vertex $v\in\Lambda$, write 
\begin{equation*}rad_v(\Lambda)=\sup_{x\in \Lambda} \delta(x,v).
\end{equation*}

It is clear that $rad_v(\mathcal{I}_t(v))\leq r$ if, and only if, $\mathcal{I}_t(v) \subset B_r(v)$. Next result tell us that the probability of $\Xi_{\Lambda,t,r}^c$  decays exponentially fast when $r\rightarrow \infty$.

\begin{lemma}\label{exponential_decay}
Let $\Lambda\subset\Z^d$ be a finite set. There are constants $c_2(d)<\infty$ and $\psi(d)>0$ such that, for all $t\in[0,1]$, 
\begin{equation*}\mathbb{P}(\xgt^{c})\leq c_2|\partial\Lambda|e^{-4{\psi(d)} r},\,\,\mbox{ for all } r>0.
\end{equation*}
\end{lemma}
\begin{proof} It suffices to show the result when $t=1$. Observe that for any vertex $x\in \Lambda^c$, we have that $x\in\mathcal{I}_1(\Lambda)$ if, and only if, $x\in\mathcal{I}_1(\partial\Lambda)$. By union bound we have
\begin{equation*}
 \mathbb{P}(\Xi_{\Lambda,r}^{c})=\mathbb{P}\left(\exists v\in\partial\Lambda\ s.t.\ rad_v(\mathcal{I}_1(v))>r\right) 
 \leq |\partial\Lambda|\max_{v\in\partial\Lambda} \mathbb{P}\left(rad_v(\mathcal{I}_1(v))>r\right).
\end{equation*}
Remember the definition of the auxiliary set $\mathcal{D}_j$ in the proof of Lemma \ref{prop}. Observe that, for any $v\in\partial\Lambda$, if $rad_v(\mathcal{I}_1(v))>r$, then, for some  $i=0,1,\dots,2d-1$, it holds 
\begin{equation}\label{galo_doido}rad_v\left(C_{\frac{i}{2d},\frac{i+1}{2d}}(u)\right)>\frac{r}{2d},\mbox{ for some $u\in\mathcal{D}_{i+1}(v)$}.
\end{equation} 
Therefore, with the convention that $\mathcal{D}_{2d}(v)=v$, we obtain
\begin{equation}\label{rad_bound}
\mathbb{P}\left(rad_v(\mathcal{I}_1(v))>r\right)\leq
\sum_{i=0}^{2d-1}\sum_{\Gamma}\mathbb{P}\left(\exists u\in \partial\Gamma\mbox{ s.t. }rad_v\left(C_{\frac{i}{2d},\frac{i+1}{2d}}(u)\right)>\frac{r}{2d}\middle|\mathcal{D}_{i+1}(v)=\Gamma\right)\P\left( \mathcal{D}_{i+1}(v)=\Gamma\right).
\end{equation} 
Recall the definition of the random set $X_i(\Gamma)$ in \eqref{fronteira}. We now use an analogous argument as in the proof of Lemma \ref{prop}. Observe that the first factor in each term on the r.h.s. of \eqref{rad_bound} can be written as
\begin{align*}\mathbb{E}\Big[\mathbb{P}\Big(\exists u\in \partial\Gamma\mbox{ s.t.} &\left.\left. rad_u\left(C_{\frac{i}{2d},\frac{i+1}{2d}}(u)\right)>\frac{r}{2d}\middle|\mathcal{D}_{i+1}(v)=\Gamma, X_i(\Gamma)\right)\right]\\
&\leq\mathbb{P}\left(\exists u\in \partial\Gamma\mbox{ s.t. }rad_u\left(C_{\frac{i}{2d},\frac{i+1}{2d}}(u)\right)>\frac{r}{2d}\middle|\mathcal{D}_{i+1}(v)=\Gamma, X_i(\Gamma)=\partial^e\Gamma\right)\\
&\leq \mathbb{P}\left(\exists u\in \partial^s\Gamma\mbox{ s.t. }rad_u\left(C_{\frac{i}{2d},\frac{i+1}{2d}}^{\Gamma^c}(u)\right)\geq\frac{r}{2d} \middle|\mathcal{D}_{i+1}(v)=\Gamma, X_i(\Gamma)=\partial^e\Gamma\right)\\
&=\mathbb{P}\left(\exists u\in \partial^s\Gamma\mbox{ s.t. }rad_u\left(C_{\frac{i}{2d},\frac{i+1}{2d}}^{\Gamma^c}(u)\right)\geq\frac{r}{2d} \right)\leq|\partial^s\Gamma|\sup_{u\in\partial^s\Gamma}\mathbb{P}\left(rad_u\left(C_{\frac{i}{2d},\frac{i+1}{2d}}(u)\right)\geq\frac{r}{2d}\right)\\
&\leq |\partial^s\Gamma|2d(2d-1)^{\lfloor{(r/2d)}\rfloor-1}\left(\frac{1}{2d}\right)^{\lfloor{r/2d}\rfloor}.
\end{align*}
Last inequality follows by a standard counting argument for Bernoulli percolation of parameter $1/{2d}$.
Therefore, plugging the above bound in \eqref{rad_bound} we obtain
\begin{align*}\mathbb{P}\left(rad_v(\mathcal{I}_1(v))>r\right)&\leq \frac{4d^2}{2d-1} \left(\frac{2d-1}{2d}\right)^{\frac{r}{2d}}\displaystyle\sum_{i=0}^{2d-1}\sum_{\Gamma}2d|\Gamma|\mathbb{P}(\mathcal{D}_{i+1}(v)=\Gamma)\\
&\leq \frac{4d^2}{2d-1} \left(\frac{2d-1}{2d}\right)^{\frac{r}{2d}}\displaystyle\sum_{i=0}^{2d-1}2d\mathbb{E}[\mathcal{D}_{i+1}(v)]\leq \frac{16d^4}{2d-1} \left(\frac{2d-1}{2d}\right)^{\frac{r}{2d}}\mathbb{E}[\mathcal{I}_1(v)].
\end{align*}
An application of Lemma \ref{prop} gives
\begin{equation*}\mathbb{P}\left(rad_v(\mathcal{I}_1(v))>r\right)\leq c_2(d)e^{-4\psi r},
\end{equation*}
with
$\psi(d)=-\frac{1}{8d}\log\frac{2d-1}{2d}$.
\end{proof}

We are ready to prove Theorem \ref{main}. The proof consists of an application of Lemma \ref{independencia} and Lemma \ref{exponential_decay}. 

\bigskip

\hspace{-0.5cm}\textit{Proof of Theorem \ref{main}.} Let $\delta=\delta(\Lambda_1,\Lambda_2)$, and write $\Xi_i=\Xi_{\Lambda_i,t,\frac{\delta}{4}}$, $i=1,2$. Also,
let $\mathcal{A}_i=\omega_t^{-1}(A_i)$ denote the set of constraints and clocks that induce the occurrence of the event $A_i $ at time $t$. For a cleaner notation, write $\mu_{\rho}=\mathbb{P}\times\mathbb{P}_{\rho}$. Apply Lemma \ref{independencia} and Lemma \ref{exponential_decay} to obtain

\begin{equation*}|\mathbb{P}_{\rho,t}(A_1\cap A_2)-\mathbb{P}_{\rho,t}(A_1)\mathbb{P}_{\rho,t}(A_2)|\leq |\mu_{\rho}\prt{\mathcal{A}_1\cap\mathcal{A}_2\cap\Xi_1\cap\Xi_2}-\mu_{\rho}(\mathcal{A}_1)\mu_{\rho}(\mathcal{A}_2)| +\mathbb{P}\prt{\Xi_1^c\cup\Xi_2^c}
\end{equation*}
\begin{equation*}
\leq|\mu_{\rho}\prt{\mathcal{A}_1\cap\Xi_1}\mu_{\rho}\prt{\mathcal{A}_2\cap\Xi_2}-\mu_{\rho}(\mathcal{A}_1)\mu_{\rho}(\mathcal{A}_2)|+c_2(|\partial\Lambda_1| +|\partial\Lambda_2|)e^{-\psi \delta}.
\end{equation*}
Observe that
\begin{equation*}\mu_{\rho}\prt{\mathcal{A}_1\cap\Xi_1} \mu_{\rho}\prt{\mathcal{A}_2\cap\Xi_2}-\mu_{\rho}\prt{\mathcal{A}_1} \mu_{\rho}\prt{\mathcal{A}_2}
\end{equation*}
\begin{equation*}=\mu_{\rho}\prt{\mathcal{A}_1\cap\Xi_1}\left[\mu_{\rho}\prt{\mathcal{A}_2\cap\Xi_2}-\mu_{\rho}\prt{\mathcal{A}_2}\right]-\mu_{\rho}\prt{\mathcal{A}_1\cap\Xi_1^c}\mu_{\rho}\prt{\mathcal{A}_2}.
\end{equation*}
This gives
\begin{equation*}|\mu_{\rho}\prt{\mathcal{A}_1\cap\Xi_1} \mu_{\rho}\prt{\mathcal{A}_2\cap\Xi_2}-\mu_{\rho}\prt{\mathcal{A}_1} \mu_{\rho}\prt{\mathcal{A}_2}|
\end{equation*}
\begin{equation*}\leq|\mu_{\rho}\prt{\mathcal{A}_2\cap\Xi_2}-\mu_{\rho}\prt{\mathcal{A}_2}|+|\mu_{\rho}\prt{\mathcal{A}_1\cap\Xi_1^c}\mu_{\rho}\prt{\mathcal{A}_2}|\leq \mathbb{P}(\Xi_1^c)+\mathbb{P}(\Xi_2^c).
\end{equation*}
A second application of Lemma \ref{exponential_decay} gives the result with $c_3=2c_2$.
\qed

\begin{remark}\label{main_cor}
The proof of Theorem \ref{main} can be extended to the following setting. Let $\Lambda_1$ and $\Lambda_2$ be disjoint subsets of $\Z^d$ with $|\partial\Lambda_1|<\infty$ and $|\partial\Lambda_2|=\infty$, and such that $\delta=\delta(\Lambda_1,\Lambda_2)$. If we assume $A_i\in \mathcal{H}_{\E(\Lambda_i)}$, $i\in\{1,2\}$,  then
\begin{equation*}\left|\mathbb{P}_{\rho,t}(A_1\cap A_2)-\mathbb{P}_{\rho,t}(A_1)\mathbb{P}_{\rho,t}(A_2)\right|\leq c_3\left(|\partial\Lambda_1|+|\partial B_{\delta}(\Lambda_1)|\right)e^{-\psi \delta},
\end{equation*}
for all $\rho=(\rho_0,\dots,\rho_{2d-1})$ and $t\in[0,1]$.
\end{remark}

\section{Continuity for Local Events}\label{contin_sec}

Consider the model on $\LL^d$. In this section we show that $\mathbb{P}_{\rho,t}(A)$ is a continuous function of $(\rho,t)$ for any local event $A$, a fact that will be useful in Section \ref{main_thm} when we prove Theorem \ref{PhaseTransition}.   

To prove continuity of $\mathbb{P}_{\rho,t}$ we shall compare two Constrained-degree percolation processes with densities $\rho$ and $\pi$ in a neighborhood of $t$, and we do this through the following coupling. Let $X=\{X(v)\}_{ v\in\Z^d}$ be a collection of independent uniform random variables on $[0,1]$, with corresponding product measure $\mathcal{Q}$.  Given $\rho=(\rho_0,\dots,\rho_{2d-1})$, write $\bar{\rho}_m=\sum_{i=0}^m \rho_i$, and define $\kappa_{\rho}$  as 
\begin{equation*}\label{coup}
 \kappa_{\rho}(v)=
    \begin{cases}
            0, &         \text{ if } X(v)<\rho_{0},\\
            j, &          \text{ if } \bar{\rho}_{j-1}\leq X(v)<\bar{\rho}_{j},\, j=1,\dots,2d-1.
           
    \end{cases}
    \end{equation*}
We clearly have, for all $v\in\Z^d$,
\begin{equation*}\P_{\rho}(\kappa_v=j)=
\mathcal{Q}(\kappa_{\rho}(v)=j),\  \forall\rho=(\rho_0,\dots,\rho_{2d-1}).
\end{equation*}
 
Given the sequences $X=\{X(v)\}_{v\in\Z^d}$ and $U=\{U_e\}_{e\in \E^d}$, we denote by $\omega_{\rho,t}(X,U)$ the induced percolation configuration. Finally, given $S\subset \Z^d$, $r\in\mathbb{N}$, and $\alpha>0$, write
\begin{align*}
W_{\rho,t,\alpha}(S)&=\{(X,U)\colon |X(v)-\bar{\rho}_{j}|>\alpha,\ \forall\ 0\leq j\leq 2d-1,\ \forall\ v\in B_{r+1}(S)\mbox{ and }\\ 
                           & U_e\notin(t-\alpha,t+\alpha),\ \forall\, e\in \E(B_{r+1}(S))\}.    
\end{align*}
Note that the set $W_{\rho,t,\alpha}(S)$ also depends on $r$,  but we keep this dependence implicit to lighten the notation. 

\begin{thm}\label{continuidade}
Let $A$ be a local event. Then $\mathbb{P}_{\rho,t}(A)$ is a continuous function of $(\rho,t)$.
\end{thm}

\begin{proof} Fix $\rho=(\rho_0,\dots,\rho_{2d-1})$, $t\in[0,1]$, and $\epsilon>0$.  Take  a finite set $K\subset\E^d$ such that $A\in\mathcal{H}_K$, and let $S=S(K)$ be the set of vertices incident to some element of $K$. Since $\mathcal{I}_t(S)$ is almost surely finite, the probability of the events $\Xi_{S,t,r}^c$ decrease to zero when $r$ goes to infinity. Fix $r$ large enough so that $P(\Xi_{S,t,r}^c)<\epsilon/2$. Note that $W^c_{\rho,t,\alpha}(S)$ also decreases when $\alpha$ goes to zero, and hence we can find $\alpha=\alpha(\rho,t,\epsilon,r)$ small enough so that $(\mathcal{Q}\times \mathbb{P})(W_{\rho,t,\alpha}^c(S))<\epsilon/2$. Denote by $\mathbb{E}_{\mathcal{Q}\times \mathbb{P}}(\cdot)$ the expectation with respect to the probability measure $\mathcal{Q}\times \mathbb{P}$. Clearly,
\begin{align*}
    |\P_{\pi,\tau}(A)-\mathbb{P}_{\rho,t}(A)|&= 
    \left|\mathbb{E}_{\mathcal{Q}\times \mathbb{P}}\left[\mathbbm{1}_{\{ \omega_{\pi,\tau}\in A\}}-
    \mathbbm{1}_{\{\omega_{\rho,t}\in A\}}\right]\right|\\
    &=\left|\mathbb{E}_{\mathcal{Q}\times \mathbb{P}}\left[\mathbbm{1}_{\{ \omega_{\pi,\tau}\in A\}}-
    \mathbbm{1}_{\{\omega_{\rho,t}\in A\}}; W_{\rho,t,\alpha}(S)\cap\Xi_{S,t,r}\right]\right|\\
    &+\left|\mathbb{E}_{\mathcal{Q}\times \mathbb{P}}\left[\mathbbm{1}_{\{ \omega_{\pi,\tau}\in A\}}-
    \mathbbm{1}_{\{\omega_{\rho,t}\in A\}};W_{\rho,t,\alpha}^c(S)\cup\Xi_{S,t,r}^c\right]\right|.
\end{align*}

 The crucial observation is the following: if $\pi=(\pi_0,\dots,\pi_{2d-1})$ and $\tau$ are such that $\displaystyle\max_{j}|\rho_{j}-\pi_j|<\alpha$, and $\tau\in(t-\alpha,t+\alpha)$, then
 \begin{equation*}\mathbb{E}_{\mathcal{Q}\times \mathbb{P}}\left[\mathbbm{1}_{\{ \omega_{\pi,\tau}\in A\}}-
    \mathbbm{1}_{\{\omega_{\rho,t}\in A\}}; W_{\rho,t,\alpha}(S)\cap\Xi_{S,t,r}\right]=0.
    \end{equation*}
    This follows because, on the event $W_{\rho,t,\alpha}(S)\cap\Xi_{S,t,r}$, the indicators $\mathbbm{1}_{\{ \omega_{\pi,\tau}\in A\}}$ and 
    $\mathbbm{1}_{\{\omega_{\rho,t}\in A\}}$ are equal.
     Indeed, let $e=\langle x,y\rangle \in \E(B_{r+1}(S))$ and note that, if the clock of edge $e$ rings after time $t+\alpha$, then it also rings after time $\tau$. Therefore, $e$ is closed on configurations $\omega_{\rho,t}$ and $\omega_{\pi,\tau}$. On the other hand, if the clock of edge $e$ rings before time $t-\alpha$, then since the constraints of $x$ and $y$ are the same on $W_{\rho,t,\alpha}(S)$, the state of $e$ is also the same in both $\omega_{\rho,t}$ and $\omega_{\pi,\tau}$. The fact that we are also in $\Xi_{S,t,r}$ enables us to use Lemma \ref{independencia}, and the claim follows.

The discussion above yields 
    \begin{equation*}
    |\mathbb{P}_{\pi,\tau}(A)-\mathbb{P}_{\rho,t}(A)|\leq (\mathcal{Q}\times \mathbb{P})(W_{\rho,t,\alpha}^c(S))+(\mathcal{Q}\times \mathbb{P})( \Xi_{S,t,r}^c)<\epsilon,
\end{equation*}
and the proof is complete.
\end{proof}

\begin{remark} The results of Section \ref{decoup_sec} and Section \ref{contin_sec} hold in the more general case of an infinite, connected and bounded degree graph $\mathcal{G}=(\V,\E)$. We chose to work with $\mathbb{L}^d$ to avoid unnecessary complications and distractions.  
\end{remark}

\section{Proof of the Main Result}\label{main_thm}

Consider the model on $\LL^2$. In this section, we prove Theorem \ref{PhaseTransition} using a coarse-graining argument. We do this in two steps, namely, the triggering argument and the induction step. In the former,  we show that the probability of connecting the box of radius $N$ to the boundary of the box of radius $2N$ by a closed dual-path decays exponentially fast when $N$ goes to infinity when $\rho_3$ and $t$ are equal to 1. For a fixed scale $L_0$, Theorem \ref{continuidade} is then used to show the existence of some $\eta=\eta(L_0)>0$ such that the probability above is still small for all values of $\rho_3$ and $t$ larger than $1-\eta$.  In the latter, we show that whenever the probability of the event just described is small at some scale, then it is also small at larger scales.  At this point, the decoupling statement of Theorem \ref{main} will be important to deal with dependence issues. After that, the renormalization takes place, and we show, through a Peierls argument, that percolation occurs almost surely.

Let $(\mathbb{\Z}^2)^*=\mathbb{\Z}^2 + (\frac{1}{2},\frac{1}{2})$ be the set of vertices of the dual lattice   $(\mathbb{\LL}^2)^*$. For each edge $e\in\mathcal{E}^2$ we define its dual $e^*$ to be the unique edge of $(\mathcal{E}^2)^*$ that crosses $e$. Given a configuration $\omega_t$ on $\mathbb{L}^2$, we define the induced dual configuration $\omega^*_t$ on $( \mathbb{L}^2)^*$, declaring each dual edge $e^*$ to be \textit{t-open} if, and only if, $e$ is \textit{t-open}, that is, $\omega^*_{t}(e^*)=\omega_{t}(e)$. We denote by $\mathbb{P}_{\rho,t}^*$ the law of the process in the dual lattice.
 
Given a natural number $N$, and a vertex $x\in(\Z^2)^*$, let $A_N(x)$ be the event that there is a closed dual path connecting the box $B_N^*(x)$ to the  boundary of the box $B_{2N}^*(x)$. Here $B_N^*(x)=[x-N,x+N]^2$ denotes the box of side length $2N$ centered at $x$. We write
\begin{equation*}
\mathbb{P}_{\rho,t}^*(N,x)=\mathbb{P}_{\rho,t}^*(A_N(x)).
 \end{equation*}
When $x$ is the origin of $(\LL^2)^*$, we adopt a lighten notation and write $A_N$, $B^*_N$, and  $\mathbb{P}_{\rho,t}^*(N)$.

\subsection{The Multiscale Scheme}

In this section, we implement a renormalization scheme. The following two lemmas give the triggering statement and the induction step, respectively.

\begin{lemma}\label{step2}
There are constants $0<c_6<1$ and $c_7>0$ with the property that, for any  $N\in\N$, there exists $\eta=\eta(N)>0$  such that 
\begin{align*}
    \mathbb{P}_{\rho,t}^*(N)\leq c_7N^2c_6^N,
\end{align*}
for all $\rho=(\rho_0,\rho_1,\rho_2,\rho_3)$ and $t$ with $(\rho_3,t)\in(1-\eta,1]^2$.
\end{lemma}

\begin{proof}
Let $\gamma=(x_0,x_1,\dots,x_l)$ with $\langle x_i,x_{i+1}\rangle \in (\mathcal{E}^2)^*$, $0\leq i \leq l$, be a path in the dual lattice and denote by $n(\gamma)$ the number of edges of $\gamma$ (in this case $n(\gamma)=l$). Also denote by $r(\gamma)$ the number of sides of the path $\gamma$, that is, 
\begin{equation*}
r(\gamma)=|\{i\in\{0,\dots,l-1\}\colon\ x_i-x_{i-1}\neq x_{i+1}-x_{i}\}|.
\end{equation*}
Let $\Gamma_{n,r,m}(x)$ be the set of all paths starting in $x\in (\mathbb{Z}^2)^*$  with $n$ edges, $r$ sides, and $m$ sides of length one. Define also  $\Gamma_{n}(x)=\bigcup_{r}\bigcup_m\Gamma_{n,r,m}(x)$. When $x$ is the origin, we write $\Gamma_{n,r,m}$ and $\Gamma_n$ instead of $\Gamma_{n,r,m}(0)$ and $\Gamma_n(0)$, respectively. Note that
\begin{equation}\label{UnionBound1}
\mathbb{P}_{\rho,t}^*(N)\leq\sum_{x\in B_N^*}\sum_{n\geq N}\mathbb{P}_{\rho,t}^*(\exists \gamma\in\Gamma_{n}(x) \mbox{ closed})\leq 8N \sum_{n\geq N}\mathbb{P}_{\rho,t}^*(\exists \gamma\in\Gamma_{n} \mbox{ closed}).
\end{equation}

Take $\rho_3=1$, $t=1$, and denote $\rho^{\star}=(0,0,0,1)$. In this case, each closed path $\gamma\in\Gamma_{n,r,m}(x)$ has at least two edges on each of its sides, with the possible exception of its extremities. Union bound gives
\begin{align}\label{UnionBound2}
\mathbb{P}_{\rho^{\star},1}^*(\exists \gamma\in\Gamma_{n} \mbox{ closed} )&\leq \sum_{m=0}^2\mathbb{P}_{\rho^{\star},1}^*(\exists \gamma\in\Gamma_{n} \mbox{ closed}, \gamma \mbox{ has } m \mbox{ sides of length one}).
\end{align}

It follows from Lemmas 1 and 2 of \cite{LSSST} that 
\begin{equation*}\label{EntropyEnergy}
\mathbb{P}_{\rho^{\star},1}^*(\exists \gamma\in\Gamma_{n,r,0} \mbox{ closed} )\leq 2^{r}\binom{n-r-1}{r-1} \left(\frac{2}{3}\right)^{n-2r}\nu^{2r},
\end{equation*}
where $\nu =\left(\frac{439}{18144}\right)^{\frac{1}{4}}$. This gives
\begin{equation}\label{JuntouUmasCoisas}
\sum_{n\geq N}\sum_{r=1}^{n/2} \mathbb{P}_{\rho^{\star},1}^*(\exists \gamma\in\Gamma_{n,r,0} \mbox{ closed})
\leq \sum_{n\geq N}\left( \frac{2}{3}\right)^n\sum_{r=1}^{\frac{n}{2}}\binom{n-r-1}{r-1} \left(\frac{9\nu^2}{2}\right)^r.
\end{equation}
A careful analysis of the function 
\begin{equation*}f(r)=\left(\frac{9\nu^2}{2}\right)^r\binom{n-r-1}{r-1},\ 1\leq r\leq\frac{n}{2},
\end{equation*}
shows that its maximum is attained at
 \begin{equation*}\label{max}
 r^*=\left\lceil \frac{n}{2}-\frac{2s +1 +\Pi^{1/2}}{8s +2}\right\rceil,
 \end{equation*}
with $s=\frac{9\nu^2}{2}$ and $\Pi=(4s +1)n(n-2)+(2s +1)^2$ (see Proposition 2 in \cite{LSSST}). Replacing $r$ with  $r^*$ on the r.h.s. of \eqref{JuntouUmasCoisas}, we obtain 
\begin{align}\label{series}
\sum_{n\geq N}\sum_{r=1}^{n/2} \mathbb{P}_{\rho^{\star},1}^*(\exists \gamma\in\Gamma_{n,r,0} \mbox{ closed})
&\leq \frac{1}{2}\sum_{n\geq N}\left( \frac{2}{3}\right)^nn \binom{n-r^*-1}{r^*-1}s^{r^*}.
\end{align}
An application of Stirling's formula gives
\begin{equation}\label{Stirling}
\binom{n-r^*-1}{r^*-1}< c_4\frac{(n-r^*-1)^{n-r^*-1}}{(r^*-1)^{r^*-1}(n-2r^*)^{n-2r^*}},
\end{equation}
for some constant $c_4>0$, and
it is not hard to see that
\begin{equation}\label{limite}
\lim_{n\rightarrow\infty}\left[c_4\frac{(n-r^*-1)^{n-r^*-1}}{(r^*-1)^{r^*-1}(n-2r^*)^{n-2r^*}}s^{r^*}\right]^{\frac{1}{n}}=\frac{1}{2}+\sqrt{s+\frac{1}{4}}.
\end{equation}
Hence, by \eqref{Stirling} and \eqref{limite}, we can find 
a constant $c_5>0$ such that
\begin{equation}\label{expo}
\left(\frac{2}{3}\right)^n\binom{n-r^*-1}{r^*-1}s^{r^*}<c_5\left(          \frac{1}{3}+\frac{2}{3}\sqrt{s+\frac{1}{4}}   +0.001  \right)^n,\ \forall n\geq1.
\end{equation}
Denote by $1/2<c_6<1$ the number between parentheses on the r.h.s. of \eqref{expo}. Plugging $(\ref{expo})$ on $(\ref{series})$, we get
\begin{equation*}
   \sum_{n\geq N}\sum_{r=1}^{n/2} \mathbb{P}_{\rho^{\star},1}^*(\exists \gamma\in\Gamma_{n,r,0} \mbox{ closed})<c_5\sum_{n\geq N}nc_6^n\leq \frac{c_5}{(1-c_6)^2} Nc_6^N.
\end{equation*}

To handle the second and third terms on the r.h.s. of \eqref{UnionBound2}, observe that the sides of length one must be necessarily on the extremities of the path. Hence,

\begin{equation*}\sum_{n\geq N}\mathbb{P}_{\rho^{\star},1}^*(\exists \gamma\in\Gamma_{n} \mbox{ closed }, \gamma \mbox{ has 1 side of length one })\leq 7 \frac{c_5}{(1-c_6)^2}(N-1)c_6^{N-1},
\end{equation*}
and 
\begin{equation*}\sum_{n\geq N}\mathbb{P}_{\rho^{\star},1}^*(\exists \gamma\in\Gamma_{n} \mbox{ closed }, \gamma \mbox{ has 2 sides of length one })\leq 12 \frac{c_5}{(1-c_6)^2}(N-2)c_6^{N-2}.
\end{equation*}
Putting all together we obtain 
\begin{equation*}\mathbb{P}_{\rho^*,t}^*(N)<  c_7N^2c_6^N,
\end{equation*} 
where $c_7=\frac{8c_5}{(1-c_6)^2}\left(1+\frac{7}{c_6}+\frac{12}{c_6^2}\right)$.
To finish the proof, use the continuity property given by Theorem \ref{continuidade} to obtain $\eta=\eta(N)>0$ such that 
\begin{equation*}\mathbb{P}_{\rho,t}^*(N)<c_7N^2c_6^N,
\end{equation*}
for all $\rho$ and $t$ such that $(\rho_3,t)\in(1-\eta,1]^2$.

\end{proof}

Let us dive into the details of the renormalization scheme. Define recursively the sequence of scales $\{L_k\}_{k\in\N}$ by 
\begin{equation}\label{scale}
L_k=\lfloor L_{k-1}^{1/2} \rfloor L_{k-1},\mbox{ for all $k\geq1$ }.
\end{equation}
We choose $L_0\geq 25$ large enough such that, for all $L\geq L_0$, the following conditions hold:

\begin{enumerate}
\item[C-1.] $Le^{-\psi L}\leq L^{-8}$;
\item[C-2.] $32(20c_3+1)L^{-1}\leq 1$;
\item[C-3.] $c_7L^2c_6^{L}\leq L^{-4}$.
\end{enumerate}

The next lemma gives the induction step.

\begin{lemma}\label{step1}For every $k\in \mathbb{Z}_{+}$, it holds 
\begin{equation*}
    \mathbb{P}_{\rho,t}^*(L_k)\leq L_k^{-4}\Rightarrow  \mathbb{P}_{\rho,t}^*(L_{k+1})\leq L_{k+1}^{-4},
\end{equation*}
for all $\rho=(\rho_0,\rho_1,\rho_2,\rho_3)$ and $t\in[0,1]$. 

\end{lemma}
\begin{proof}
First, we cover the boundary of the box $B_{L_{k+1}}^*$ with boxes $B_{L_{k}}^*(x)$, in such a way that
$x$ lies in the boundary of $B^*_{L_{k+1}}$, and for any $B^*_{L_{k}}(y)$ in the coverage, it holds that $|B^*_{L_{k}}(x)\cap B^*_{L_{k}}(y)\cap \partial B^*_{L_{k+1}}|$ is at most 1. We cover the boundary of $B_{2L_{k+1}}^*$ in a similar way.

 An upper bound for the probability of the event $A_{k+1}$ can be obtained with the following observation. Whenever $A_{k+1}$ occurs, it must be that $A_k$ occurs for some pair of boxes in the previous scale, one box in each of the coverings described above (see Figure \ref{figura} for an illustration). 
This gives an entropy factor bounded by $32 \left(\frac{L_{k+1}}{L_k}\right)^2$. Therefore, by union bound, Theorem \ref{main}, and translation invariance, we obtain 
\begin{align}
\mathbb{P}_{\rho,t}^*(L_{k+1})&\leq 32 \left(\frac{L_{k+1}}{L_k}\right)^2\mathbb{P}_{\rho,t}^*(A_k(x)\cap A_k(y))\nonumber\\ 
&\leq 32 \left(\frac{L_{k+1}}{L_k}\right)^2\left[ \mathbb{P}_{\rho,t}^*(L_k)^2+2c_3|\partial B_{L_k}|\exp\left(-\psi \delta_k(x,y)\right)\label{conta}
\right],
\end{align}
where $\delta_k(x,y)$ is the distance between boxes $B^*_{2L_k}(x)$ and $B^*_{2L_k}(y)$. Since we choose $L_0\geq 25$, it follows that
\begin{equation}\label{menor_distancia}
\min_{x,y}\delta_k(x,y)=L_{k+1}-4L_k=L_k\left(\lfloor L_{k}^{1/2} \rfloor-4\right)\geq L_k,\ \forall k\geq0.
\end{equation}
\begin{figure}
\centering
\begin{tikzpicture}[scale=0.03]
\filldraw (0,0) circle (10pt);
\draw (50,-50)node[below] {$B_{L_{k+1}}$};
\draw (100,-100)node[below] {$B_{2L_{k+1}}$};
\draw (0,80)node[below] {$B_{2L_{k}}$};
\draw (0,0)node[above] {$0$};
\draw[line width=.03cm,gray] (-100,-100) rectangle (100,100);
\draw[line width=.03cm,gray] (-50,-50) rectangle (50,50);
\begin{scope}[shift={(-33,0)}]
\begin{scope}[shift={(0,50)}]
\draw[line width=.03cm] (45,-5) rectangle (55,5);
\draw[line width=.03cm] (40,-10) rectangle (60,10);
\end{scope}
\end{scope}
\begin{scope}[shift={(13,50)}]
\begin{scope}[shift={(-33,0)}]
\begin{scope}[shift={(0,50)}]
\draw[line width=.03cm] (45,-5) rectangle (55,5);
\draw[line width=.03cm] (40,-10) rectangle (60,10);
\end{scope}
\end{scope}
\end{scope}
\draw[line width=.05cm,red] (13,50) .. controls (90,60) and (-20,70) .. (28,100);
\end{tikzpicture}\caption{An illustration of the cascade events $A_{L_k}$. \label{figura}}
\end{figure}
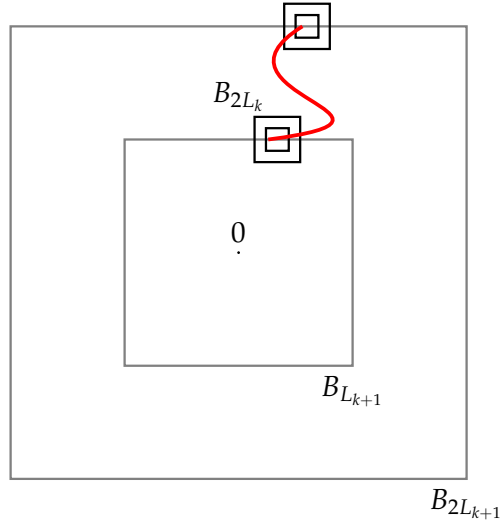
As $|\partial B_{L_k}|\leq 10 L_k$, for all $k\geq 0$, Equations \eqref{conta} and \eqref{menor_distancia} give  
\begin{equation*}
\mathbb{P}_{\rho,t}^*(L_{k+1})\leq 32 \left(\frac{L_{k+1}}{L_k}\right)^2
\left(\mathbb{P}_{\rho,t}^*(L_k)^2+ 20c_3L_k\exp(-\psi L_k)   \right).
\end{equation*}
If $\mathbb{P}_{\rho,t}^*(L_k)\leq L_k^{-4}$, last inequality implies
\begin{equation*}
\frac{\mathbb{P}_{\rho,t}^*(L_{k+1})}{L_{k+1}^{-4}}\leq 
32L_{k+1}^4L_k\left(   L_k^{-8}+20c_3L_ke^{-\psi L_k}\right),
\end{equation*}
and $(\ref{scale})$ gives 
\begin{equation*}\label{QuaseLa}
\frac{\mathbb{P}_{\rho,t}^*(L_{k+1})}{L_{k+1}^{-4}}\leq 32 L_k^7\left(   L_k^{-8}+20c_3L_ke^{-\psi L_k}\right).
\end{equation*}
By our choice of $L_0$ (condition C-1), it follows that  
\begin{equation}\label{ExpGanha}
    L_ke^{-\psi L_k}\leq L_k^{-8},\ \forall k\geq 0,
\end{equation}
and, with the aid of \eqref{ExpGanha}, we get
\begin{equation*}\frac{\mathbb{P}_{\rho,t}^*(L_{k+1})}{L_{k+1}^{-4}}\leq
32(20c_3+1) L_k^{-1}.
\end{equation*}
The result follows, again by our choice of $L_0$ (condition C-2).

\end{proof}

This completes the two key lemmas. The proof of Theorem \ref{PhaseTransition} now follows with the help of the well known topological fact that a cluster of open edges on $\LL^2$ is finite if, and only if, it is surrounded by a circuit of closed edges on $(\mathbb{L}^2)^*$. We write the details in the next section.  

\subsection{Proof of Theorem \ref{PhaseTransition}} 


Our first observation is that Lemma \ref{step2} and Lemma \ref{step1} imply the existence of $\eta>0$ such that 
\begin{equation}\label{multi}
    \mathbb{P}_{\rho,t}^*(L_k)\leq L_k^{-4},\ \forall  k\geq 0,
\end{equation}
for all $\rho$ and $t$ such that $(\rho_3,t)\in(1-\eta,1]^2$. Indeed, take $N=L_0$ in Lemma \ref{step2} and note that, by our choice of $L_0$ (condition C-3), 
$
  \mathbb{P}_{\rho,t}^*(L_{0})\leq  c_7L_{0}^2c_6^{L_{0}}\leq L_{0}^{-4}.
$ Then, apply Lemma \ref{step1}.

Now, fix $(\rho_3,t)$ such that (\ref{multi}) holds. Let $\mathcal{O}_L$ be the open cluster of the line segment $\{0,\dots,L\}\times\{0\}$, that is,  
\begin{equation*}\mathcal{O}_L=\bigcup_{u\in\{0,\dots,L\}\times\{0\}}\{v\in \Z^2: v \mbox{ is connected to $u$ by an open path}\}.
\end{equation*}
Clearly, $\mathcal{O}_L$ is infinite with positive probability if, and only if, $\mathcal{O}_0$ percolates with positive probability.
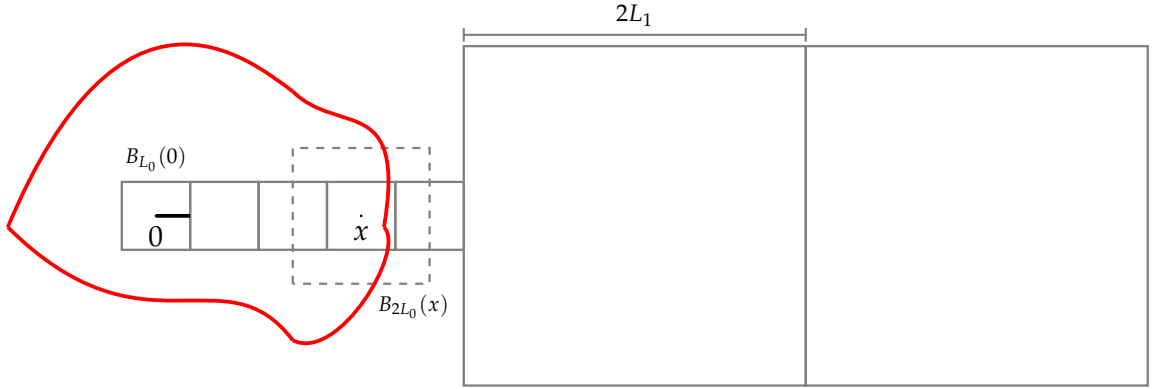
\begin{figure}[ht]
\centering
\begin{tikzpicture}[scale=0.03]
\filldraw (25,5) circle (10pt);
\draw (25,05)node[below] {$0$};
\draw (138,-25)node[below] {\scriptsize $B_{2L_{0}}(x)$};
\draw (25,40)node[below] {\scriptsize $B_{L_{0}}(0)$};
\draw (115,05)node[below] {$x$};
\filldraw (115,05) circle (10pt);

\draw[line width=.03cm,gray] (10,-10) rectangle (40,20);
\draw[line width=.03cm,gray] (40,-10) rectangle (70,20);
\draw[line width=.03cm,gray] (70,-10) rectangle (100,20);
\draw[line width=.03cm,gray] (100,-10) rectangle (130,20);
\draw[line width=.03cm,gray] (130,-10) rectangle (160,20);

\draw[line width=.03cm,gray] (160,-70) rectangle (310,80);
\draw[line width=.03cm,gray] (310,-70) rectangle (460,80);

\draw[gray, thick] (160,82) -- (160,88);
\draw[gray, thick] (310,82) -- (310,88);
\draw[gray, thick] (160,85) -- (310,85);
\draw (235,103)node[below] {\footnotesize $2L_1$};

\draw[line width=.03cm,gray,dashed] (85,-25) rectangle (145,35);
\draw[line width=0.05cm,red] (-40,0) .. controls (20,-60) and (55,-10) .. (85,-50);
\draw[line width=0.05cm,red] (85,-50) .. controls (105,-60) and (135,-10) .. (125,0);
\draw[line width=0.05cm,red] (125,0) .. controls (135,60) and (105,40) .. (85,60);
\draw[line width=0.05cm,red] (85,60) .. controls (65,75) and (10,120) .. (-40,0);
\draw[line width=0.05cm,black] (25,5) .. controls (25,5) and (25,5) .. (40,5);

\begin{scope}[shift={(-33,0)}]
\begin{scope}[shift={(0,50)}]
\draw[line width=.03cm] (10,20) rectangle (10,20);
\draw[line width=.03cm] (10,20) rectangle (10,20);
\end{scope}
\end{scope}
\end{tikzpicture}
\caption{A circuit of closed dual edges creates a path connecting the box $B_{L_0}(x)$ to the boundary of the box $B_{2L_0}(x)$ (dashed line), inducing the occurrence of $A_{L_0}(x)$.}
\label{fig_peierls}
\end{figure}
Divide $[0,1,2,\dots)$ into intervals $I_k=\{L_k,L_k+1,\dots,L_{k+1}\}$. Each of this intervals, say interval $I_k$, is partitioned into $L_{k+1}/L_k$ boxes of radius $L_k$. Let $x_{k,i}$ be the vertex at the center of the $i$-th box of the $k$-th scale. The crucial observation is that, if $|\mathcal{O}_{L_0}|<\infty$, then there exists an infinite closed dual circuit surrounding $\{0,\dots,L\}\times\{0\}$. This in turn implies the occurrence of the event $A_{L_k}(x)$ for some $x\in\displaystyle\bigcup_k\{x_{k,i}:i=1,2,\dots,\lfloor L_k^{1/2}\rfloor\}$. See Figure \ref{fig_peierls} for a sketch of this event.  Since for each $k$ there are $\frac{L_{k+1}}{L_k}$ boxes of size $L_k$, it follows that
\begin{equation*}
1-\mathbb{P}_{\rho,t}(|\mathcal{O}_{L_0}|=\infty)\leq \sum_{k\geq 0}\frac{L_{k+1}}{L_k}\mathbb{P}_{\rho,t}^*(L_k)\leq \sum_{k\geq 0}L_k^{-7/2}<1.
\end{equation*}
Last inequality follows by our choice of scales on (\ref{scale}) and the fact that $L_0\geq 25$. This completes the proof.
\qed

\section*{Acknowledgments}
The authors are grateful to two anonymous referees for their valuable comments and suggestions. The research of R\'{e}my Sanchis was partially supported by FAPEMIG (PPM 0600-16) and CNPq. Diogo dos Santos was partially supported by PNPD/CAPES. Roger Silva was partially supported by FAPEMIG (Edital Universal). This study was financed in part by the Coordena\c{c}\~{a}o de Aperfei\c{c}oamento de Pessoal de N\'{i}vel Superior, Brasil (CAPES), Finance Code 001.


\begin{thebibliography}{4}



\bibitem{ATT} Ahlberg D., Tassion V. and Teixeira A., Sharpness of the phase transition for continuum percolation in $\mathbb{R}^2$.  \textit{Probab. Theory Relat. Fields}, {\textbf {172}} (2018) 525--581.

\bibitem{AM} Amir G. and Baldasso R., Percolation in majority dynamics. \textit{Electron. J. Prob.}, {\textbf {25}} (2020) 18 pp.


\bibitem{BT3} Baldasso R. and Teixeira A., How can a clairvoyant particle escape the exclusion process.  \textit{Ann. Inst. H. Poincar\'{e} Probab. Statist.}, {\textbf {54}} (2018) 2177--2202.

\bibitem{BR} Bollob\'{a}s B. and Riordan O., The critical probability for random voronoi percolation in the plane is $\frac 12$. \textit{Probab. Theory Relat. Fields}, {\textbf {136}} (2006) 417--468.

\bibitem{BDS} Bramson M., Durrett R. and Schonmann R., The contact process in a random environment. \textit{Ann. Probab.}, {\textbf {19}} (1991) 960--983. 

\bibitem{CK} Campanino M. and Klein A., Decay of two-point functions for $(d+1)$-dimensional percolation, Ising and Potts models with $d$-dimensional disorder. \textit{Comm. Math. Phys.}, {\textbf {135}} (1991) 483--497.

\bibitem{CKP} Campanino M., Klein A. and Perez J. F., Localization in the ground state of the Ising model with a random transverse field. \textit{Comm. Math. Phys.}, {\textbf {135}} (1991) 499--515.

\bibitem{LSSST} de Lima B.N.B., Sanchis R., dos Santos D.C., Sidoravicius V. and Teodoro R., The Constrained-degree percolation model, \textit{Stoch Process Their Appl.}, {\textbf {130}} (2020) 5492--5509. 
 
\bibitem{FSZD} Furlan A.P., dos Santos D.C., Ziff R.M. and Dickman R., Jamming and percolation of dimers in restricte-valence random sequential absortion, \textit{Phys. Rev. Research}, {\textbf 2} (2020) 043027. 

\bibitem{Ga} Garet O., Marchand R. and Marcovici I., Does Eulerian percolation on $\Z^2$ percolate?, \textit{ALEA}, {\textbf {15}} (2018) 279--294.

\bibitem{GJ} Grimmett G., Janson S., Random graphs with forbidden vertex degrees, \textit{Random Struct. Alg.}, {\textbf {37}} (2010) 137--175.

\bibitem{GL} Grimmett G., Li Z., The 1-2 model, \textit{Contemp. Math.}, {\textbf {969}} (2017) 139--152.

\bibitem{HO} Hoffman C., Phase transition in dependent percolation. \textit{Commun. Math. Phys.} {\textbf {254}} (2005) 1--22.

\bibitem{HL} Holroyd A.E. and Li, Z., Constrained percolation in two dimensions, \textit{Ann. Inst. Henri Poincar\'{e} D}, {\textbf 8}(3) (2021) 323--375.

\bibitem{JMP} Jonasson J., Mossel E. and Peres Y., Percolation in a dependent random environment, \textit{Random Struct. Alg.}, {\textbf {16}} (2000) 333--343.

\bibitem{KSV} Kesten H., Sidoravicius V. and Vares M. E., Oriented percolation in a random environment, \emph{arXiv:1207.3168}.

\bibitem{K} Klein A., Extinction of contact and percolation processes in random environment, \textit{Ann. Probab.}, {\textbf {22}} (1994) 1227--1251.

\bibitem{ZL} Li Z., Constrained percolation, Ising model, and XOR Ising model on planar lattices. \textit{Random Struct. Alg.}, {\textbf {57}} (2020) 474--525.


\bibitem{MW} McCoy B.M. and Wu T.T., Theory of a two-dimensional Ising Model with random impurities. I. Thermodynamics., \textit{Phys. Rev.} {\textbf {176}} (1968) 631--643.


\bibitem{Te} Teodoro R., Constrained-degree Percolation, PhD Thesis, IMPA, 2014.






\end{thebibliography}
\end{document}